\documentclass[a4paper,leqno]{article}
\usepackage[a4paper,centering,vscale=0.77,hscale=0.77]{geometry}\usepackage{amsmath,amsthm,amssymb,bbm,bm}
\usepackage{xcolor}
\usepackage{verbatim}
\usepackage{todonotes}
\usepackage{mathrsfs}

\usepackage{thmtools}
\usepackage[pdfdisplaydoctitle,colorlinks,breaklinks,urlcolor=blue,linkcolor=blue,citecolor=blue]{hyperref}
\usepackage{enumitem}

\newtheorem{theorem}{Theorem}[section]
\newtheorem{lemma}[theorem]{Lemma}
\newtheorem{proposition}[theorem]{Proposition}

\newtheorem{definition}[theorem]{Definition}
\newtheorem{remark}[theorem]{Remark}

\newcommand{\E}{\mathbb E}

\begin{document}

\title{Long time behavior of the stochastic 2D Navier-Stokes equations }
\author{Benedetta Ferrario\thanks{Dipartimento di Scienze Economiche e Aziendali, Universit\`a di Pavia, 27100 Pavia, Italy.
E-mail: \texttt{benedetta.ferrario@unipv.it}}
\and Margherita Zanella\thanks{Department of Mathematics, Politecnico di Milano,
Via E.~Bonardi 9, 20133 Milano, Italy.
E-mail: \texttt{margherita.zanella@polimi.it}}}

\maketitle

\begin{abstract}
We  review some  basic results on existence and  uniqueness of the invariant measure for the two-dimensional stochastic Navier-Stokes equations. A large part of the literature concerns the additive noise case; after revising these models, 
we consider our recent result \cite{FerZanSNS} with a  multiplicative noise.
\end{abstract}


\tableofcontents

\section{Introduction}

The Navier-Stokes equations 
describe the motion of  homogeneous incompressible viscous fluids; they are
\begin{equation}\label{eq-fluid}
\partial_t u -  \nu \Delta u +(u \cdot \nabla)u +\nabla p =f ; 
\qquad \mbox{div }u =0 
\end{equation}
where $u=u(t,\xi)$ and $p=p(t,\xi)$ are the 
velocity vector  and the (scalar) pressure, respectively, defined for 
$t\ge 0$ and $\xi \in \mathcal{O}\subseteq \mathbb R^d$ ($d=2$ or $d=3$);
 $\nu>0$ is the kinematic viscosity parameter.
Suitable initial and boundary conditions are given.
In the right-hand-side, the forcing term can be deterministic  and/or  stochastic.

In this paper we review some results for the bidimensional Navier-Stokes equations with both  forcing terms, deterministic and stochastic.
The literature is quite huge and we cannot quote all the contributions. Our purpose is to present and compare some results obtained in the last thirty years on the stationary solutions, called invariant measures in this setting. 
In particular, the results for the uniqueness of the invariant measure. 
Under some assumptions on the  forcing terms,  we will show that 
there exists a unique invariant measure even when the deterministic Navier-Stokes equations \eqref{eq-fluid}
have more than one stationary solution. 

In the deterministic setting, uniqueness of the stationary solutions  occurs  when the data are small enough, or the viscosity is large enough (see \cite{Temam2001} Chapter 2).
Roughly speaking, a strong enough dissipation prevents the existence of different stationary solutions. 
Moreover, when the stationary solution is unique then the system converges to it as time diverges to $+\infty$. 

We point out that  the behavior of the motion of fluids with small viscosity is a interesting and challenging problem,   also in connection with the vanishing viscosity limit leading to the equation of motion of inviscid fluids (i.e., the Euler equations).

The idea proposed by  Kolmogorov (see \cite{Frisch,Gall,VisFurk} and reference therein)  in his K41 theory of turbulence 
is to introduce a noisy forcing term in the equation of motion so to mix up the dynamics and break the possibility of different asymptoptic behaviors. 
In this setting a statistical description  is given, as usually done when describing the chaotic behaviour of the so called turbulent flows. 
By means of an invariant (stationary) probability measure, when it exists and is unique, one can  describe the asymptotic behavior of a turbulent fluid from a statistical point of view.
Therefore the stochastic Navier-Stokes equations are an important model 
in the turbulence theory of fluids. 
 
Starting from this idea, many researchers have studied the Navier-Stokes equations with a stochastic forcing term. 
In this paper we revise some contributions for the Navier-Stokes equations in a smooth bounded domain of $\mathbb R^2$, where in the right-hand-side of the equation of motion \eqref{eq-fluid} we introduce a noisy forcing term $G(u){\rm d}W$. 
It represents a noise which is white in time and colored in space; moreover it might depend  on the velocity (the so called multiplicative noise) or not (the so called additive noise).  
We will focus mainly on our recent result, when the noise  is multiplicative  and of at most linear growth. 
But we will revise also the main results for the case of additive noise, as well.  In  the literature there are also  different models with a  random force   acting at discrete times (see, e.g., \cite{BKL01,KS_BOOK}).

As far as the structure of the paper is concerned, in Section \ref{setting_sec} we introduce the notations and   assumptions. 
In Section  \ref{S:wellpos} sufficient conditions for the existence of a unique solution to the Navier-Stokes equations are given.
The existence of  invariant measures is considered in  Section  \ref{S:ex-inv-meas}, the uniqueness  in Section \ref{S:unique-meas} and the asymptotic stability in Section \ref{sec_asy_sta}.

\section{Mathematical setting}
\label{setting_sec}
 In this section we fix the notations, explain the assumptions, formulate the framework of our problem and state the main results.
We refer mainly to the book \cite{Temam2001} by Temam for the deterministic Navier-equations and to the books \cite{DapZab2,DapZab1} by Da Prato and Zabczyk for the SPDE's.

Let $\mathcal{O}$ be a smooth bounded domain in       $\mathbb R^2$.
We define
\[
{\cal V} = \{ u=(u_1,u_2) \in [C_{0}^{\infty}(\mathcal{O})]^{2}: \mbox{div }u = 0 \}
\]
and take the closure of this space in $[L^{2}({\cal
O})]^{2}$ and $[H_{0}^{1}({\mathcal{O}})]^{2}$; we obtain, respectively, the spaces 

\[
H=\{ u \in [L^{2}({\mathcal{O}})]^{2}: 
        \mbox{div }u=0, u \cdot n =0 \mbox{ in }\partial   {\mathcal{O}} \}
\]
and 
\[
V=\{u \in [H^{1}_{0}({\mathcal{O}})]^{2}: \mbox{div }u=0\},
\]
where by $n$ we denote the outer normal to $\partial {\mathcal{O}}$.
These are Hilbert spaces; we equip $H$ with the scalar product in $L^2$ involving  the velocity
and $V$ with scalar product  in $L^2$ involving the gradient of the velocity
\[
(u,v)_H=\sum_{i=1}^2 \int_{\cal O}u_i(\xi) v_i(\xi){\rm d}\xi,\qquad 
(u,v)_V=\sum_{i=1}^2 \int_{\cal O} \nabla u_i(\xi)\cdot \nabla v_i(\xi){\rm d}\xi.
\]
We consider the orthogonal projector $\Pi$ from  $[L^{2}({\mathcal{O}})]^{2}$ onto $H$ and the Stokes operator $A$ is defined as 
\[
Au=-\Pi \Delta u, \qquad \forall u \in 
\mathcal{D}(A)=[H^{2}({\mathcal{O}})]^{2}\cap V.
\]
We can define the fractional powers $A^{\alpha}$, for $\alpha \in \mathbb R$, and the 
space $\mathcal{D}(A^{\alpha})$ that corresponds to the Sobolev space
$[H^{2\alpha}({\cal O})]^{2}$ with the suitable boundary condition and the divergence free condition.
In particular, $V={\mathcal D}(A^{\frac 12})$ and for $\alpha>\beta$
the continuous embedding ${\cal D}(A^\alpha)\subset {\cal D}(A^\beta)$ is compact too.
Denoting by $H'$ and $V'$ the dual spaces, if we identify $H$ with
$H'$ we get the following continuous embeddings 
\[
V \subset H \subset V'
\]
and  $V^\prime={\mathcal D}(A^{-\frac 12})$.

The operator $A$ is a closed positive unbounded self-adjoint operator in $H$ with
the inverse 
$A^{-1}$ which is a 
self-adjoint compact operator in $H$. By the classical
spectral theorems there exists 
a sequence $\{\lambda_{n}\}_{n=1}^{\infty}$ 
of eigenvalues of the Stokes operator  
\[
0<\lambda_{1} \le \lambda_{2} \le \ldots, \qquad  \lim_{n \rightarrow +\infty}\lambda_{n}=+ \infty,
\]
corresponding to the eigenvectors $\{e_n\}_{n\in \mathbb{N}} \in \mathcal{D}(A)$, which form an orthonormal
basis in $H$. We will denote by $P_N$ the orthogonal projection in $H$ onto the space Span$\{e_n\}_{1\le n\le N}$.

The Poincar\'e inequality holds true, since we consider a 
bounded domain ${\cal O}$ and homogeneous Dirichlet  boundary conditions; we write it as 
\begin{equation*}
\label{lambda_1}\|u\|^2_V \ge \lambda_1 \|u\|^2_H.\end{equation*}

Now, consider the bilinear operator $B$ 
from $V \times V$ into $V'$ defined as
\[
<B(u,v),z> \:= \int_{\cal O} z(\xi) \cdot \left(u(\xi) \cdot \nabla \right)
   v(\xi) \; d\xi \qquad \forall \: u,v,z \in V
   \]
where $<\cdot, \cdot>$ denotes  the duality pairing
between $V$ and $V^\prime$.

By incompressibility condition we have
\begin{equation}\label{Buvv=0}
<B(u,v),v>\:=0
\end{equation}
\begin{equation}
<B(u,v),z>\:=-<B(u,z),v>.
\end{equation}
A result of \cite{giga}, (Lemma 2.2), allows to extend the definition of the bilinear operator to more general spaces. We  quote it  here for the particular two dimensional case.
\begin{lemma} \label{gig}
Let $0 \le \delta < 1$. Then there exists 
some constant $C=C(\rho, \theta, \delta)$ such that
\[ \|B(u,v)\|_{{\cal D}(A^{-\delta} )} \le C \:
\|u\|_{{\cal D}(A^\theta )} \; \|v\|_{{\cal D}(A^\rho )} \hspace*{2cm} \forall u \in {\cal D}(A^{\theta}), 
v \in {\cal D}(A^{\rho})
\]
where  $\theta >0,\: \rho >0,\: 
\delta + \theta + \rho  \ge 1,\: \delta + \rho > {\frac 12}$.
\end{lemma}
In particular: $B: {\cal D}(A^{\frac 14})\times {\cal D}(A^{\frac 14})\to {\cal D}(A^{-\frac 12})$ is well defined and there exists a positive constant $C$ such that
\begin{equation}\label{stima-B-L4}
\|B(u,v)\|_{V^\prime}\le C\|u\|_{{\cal D}(A^{\frac 14})} \|v\|_{{\cal D}(A^{\frac 14})}.\qquad \forall u,v \in {\cal D}(A^{\frac 14}).
\end{equation}

Projecting equation \eqref{eq-fluid} onto $H$, we get rid of the pressure term and we obtain  the following
abstract formulation
\begin{equation} \label{sns} \left\{
        \begin{array}{l}
        {\rm d}u(t) +\left[\nu Au(t)+B(u(t),u(t))\: \right] \; {\rm d}t=f (t)\; {\rm d}t + G(u(t)) {\rm d}W(t)
        \\
        u(0)=u_{0}.
        \end{array} \right.
\end{equation}

This problem is studied for any initial velocity $u_0\in H$ and deterministic forcing term  $f\in L^2(0,T;V^\prime)$; later one we will ask $f$ to be independent of time. 

Now we define the stochastic forcing term.
Given the Hilbert spaces $H$ and $K$, in what follows we denote by  $L(H,K)$ the space of linear operators from $H$ to $K$  and with $L_{HS}(H,K)$  the space of Hilbert-Schmidt operators from $H$ to $K$. 
Given a separable real Hilbert space $U$ with  an orthonormal basis  $(f_n)_{n\in\mathbb{N}}$, we assume that $\{W(t)\}_{t\ge 0}$ is a $U$-cylindrical Wiener process defined on a probability space $(\Omega, {\cal F}, \mathbb P)$; we consider the natural filtration  ${\cal F}_{t}$ fulfilling the standard conditions, i.e. it is  continuous  and complete.  The Wiener process can be represented as 
\begin{equation}\label{def-serie-W}
W(t)=\sum_{n=1}^\infty \beta_n(t)f_n
\end{equation}
for a sequence of independent real valued standard Wiener processes $(\beta_n)_n$.

The multiplicative term $G(u)$ can reduce to a constant  operator (the so called additive noise)  or in general it depends on the unknown $u$ (the so called multiplicative noise). We assume the following.

\begin{enumerate}[label=\textbf{(A\arabic*)}] \itemsep0.3em
\item \label{noise1} The operator 
$G:H\to L_{HS}(U;H)$
satisfies
\begin{equation*}
\|G(u)\|_{L_{HS}(U,H)}^2\le C_1 \|u\|_H^2+C_2, \qquad u \in H, 
\end{equation*}
for some non-negative constants $C_1$ and $C_2$.
\end{enumerate}
Later on we will specify more assumptions on $G$.

\subsection{Solutions}

In a two dimensional domain, the Navier-Stokes equations are well studied as far as existence, uniqueness and regularity are concerned.

The following definition of martingale solution is given.
\begin{definition}
\label{mart_sol_def}
We say that there exists a martingale solution of the Navier-Stokes equation
\eqref{sns} on the interval $[0,T]$ and with initial velocity $u_0\in H$ if there exist
a stochastic basis $(\widetilde \Omega, \widetilde{\mathcal{F}}, \widetilde{\mathbb{P}},\widetilde{\mathbb{F}})$, a $U$-cylindrical Wiener process $\widetilde W$,
and a progressively measurable process $u:[0,T]\times \widetilde \Omega \rightarrow H$ with $\widetilde{\mathbb{P}}$-a.e. path
\begin{equation}\label{reg-u-in-def}
u \in C([0,T];H) \cap L^2(0,T;V)
\end{equation}
such that $\widetilde{\mathbb{P}}$-a.s., the identity 
\begin{multline*}
(u(t), \psi)_H+\nu \int_0^t (A^{\frac 12}u(s), A^{\frac 12}\psi)_H\, {\rm d}s 
+\int_0^t \langle B(u(s),u(s)),\psi \rangle\, {\rm d}s 
\\=( u_0, \psi)_H + \int_0^t\langle f(s), \psi\rangle {\rm d}s + \langle \int_0^t G(u(s))\, {\rm d}\widetilde{W}(s), \psi\rangle
\end{multline*}
holds true for any $t \in [0,T]$, $\psi \in V$.
\end{definition}

The latter identity is equivalent to say that equation \eqref{sns} holds in the space $V^\prime$.

It is possible  to fix the stochastic basis and in this case  we speak of a strong solution (in the probabilistic sense). This is the definition.
\begin{definition}
\label{strong_sol_def}
Given a stochastic basis $(\Omega, \mathcal{F}, \mathbb{P}, \mathbb{F})$ and  a $U$-cylindrical Wiener process $W$, a strong solution of  the Navier-Stokes equation \eqref{sns} on the interval $[0,T]$ 
with initial velocity $u_0\in H$ is an $H$-valued continuous $\mathbb{F}$-adapted process $u$ with $\mathbb{P}$-a.e. path in $L^2(0,T;V)$
such that ${\mathbb{P}}$-a.s., the identity 
\begin{multline*}
(u(t), \psi)_H+\nu \int_0^t (A^{\frac 12}u(s), A^{\frac 12}\psi)_H\, {\rm d}s 
+\int_0^t \langle B(u(s),u(s)),\psi \rangle\, {\rm d}s 
\\
=( u_0, \psi)_H + \int_0^t\langle f(s), \psi\rangle {\rm d}s + \langle \int_0^t G(u(s))\, {\rm d}W(s), \psi\rangle
\end{multline*}
holds true for any $t \in [0,T]$, $\psi \in V$.
\end{definition}

\section{Well posedness}
\label{S:wellpos}

We have the following existence result of martingale solutions as given in Definition \ref{mart_sol_def} (cf. \cite{FlaGat95}).
\begin{proposition}\label{teo-esistenza-sol-mart}
Let $T$ be any  finite time. If   $f\in L^2(0,T;V^\prime)$ and $G$ fulfills Assumption \ref{noise1}, 
then for any 
 $u_0\in H$ there exists a martingale solution of the Navier-Stokes equation \eqref{sns} on the interval $[0,T]$.
\end{proposition}

The proof is based on the finite dimensional Galerkin approximation for which uniform estimates are found; they allow to pass to the limit so to recover the full (infinite dimensional) Navier-Stokes equation. The conservation of the energy (obtained thanks to property \eqref{Buvv=0}) allows to pass from a local (in time) existence result to the global one (see, e.g.,  \cite{FlaGat95} and \cite{KS_BOOK}).

Moreover a mean energy estimate holds; this is crucial for the tightness in the next section, to prove the existence of invariant measures.
We write  it here, already for the full system \eqref{sns} but one should consider it first for the finite dimensional Galerkin approximation. 
By means of It\^o formula for $\|u(t)\|_H^2$ and assuming that the constant  $C_1$ in Assumption \ref{noise1} vanishes, 
we easily get
\begin{equation}\label{mean-sq-energy}
\E\|u(t)\|_H^2+\nu \int_0^t \|A^{\frac 12}u(s)\|_H^2{\rm d}s\le
\|u_0\|_H^2+ \frac 1\nu \int_0^t \|f(s)\|^2_{V^\prime}{\rm d}s+ t C_2.
\end{equation}
An analogous estimate, but with different multiplicative constants, holds under assumption
\ref{noise1} if $C_1$ is small enough, i.e. $C_1<4\nu \lambda_1$ (cf. Lemma 3.1 in \cite{FerZanSNS}): there exist two positive constants $C_\nu$ and $C_\nu^\prime$ such that
\begin{equation}
\E\|u(t)\|_H^2+C_\nu \int_0^t \|A^{\frac 12}u(s)\|_H^2{\rm d}s\le
\|u_0\|_H^2+ C^\prime_\nu \int_0^t \|f(s)\|^2_{V^\prime}{\rm d}s+ t C_2.
\end{equation}

\begin{remark} \label{oss-splitting-flandoli}
Let us point out that it is not necessary that the operator $G(u)$ is Hilbert-Schmidt with values in $H$. For instance, in the additive case
if one gives up the mean square estimate \eqref{mean-sq-energy}, it is possible to weaken the  assumption on the noise term 
asking the range of $G$ to be  contained in the space ${\cal D}(A^{\alpha})$ for some  $\alpha>\frac 14$; see details in \cite{Fla94}.
Notice that $G \in L_{HS}(U,H)$ corresponds to the stronger condition  that the range of $G$ is  contained in the space ${\cal D}(A^{\alpha})$ for some  $\alpha>\frac 12$.
An example of such an operator is $G=A^{-\alpha}$ when $U=H$.

For $\alpha>\frac 14$, the problem is studied by introducing the auxiliary processes $z$ and $v$, where $z$ is the solution of the linear stochastic Stokes equation
\begin{equation}\label{OU-eq}
     dz(t) +\nu Az(t)\  {\rm d}t= G {\rm d}W(t); \qquad z(0)=0
\end{equation}
     
and  $v=u-z$ fulfills the random equation
\[
 \frac{dv}{dt} +\nu Av+B(v,v)+B(v,z)+B(z,v)=f -B(z,z) ; \qquad v(0)=u_{0}.
\]
Assuming that the range of $G$ is contained in ${\cal D}(A^{\alpha})$ for $\alpha>\frac 14$, it follows that a.a. paths of $z$ are in $C([0,+\infty);{\cal D}(A^{\frac 14}))$ so $B(z,z)\in C([0,+\infty);V^\prime)$. Working pathwise, by classical techniques we get  the existence of a unique $v\in C([0,T];H)\cap L^2(0,T;V)$; 
moreover by interpolation one gets $v \in L^4(0,T;{\cal D}(A^{\frac 14}))$ too. So $u=v+z\in C([0,T];H) \cap L^{4}(0,T;{\cal D}(A^{\frac 14}))$. 
The regularity of the paths of the process $u$ is enough to define the  bilinear term, thanks to \eqref{stima-B-L4}, and $B(u,u)\in L^2(0,T;V^\prime)$. Hence  
 the equation \eqref{sns} holds in the space $V^\prime$.
But  the energy estimate now holds pathwise for $v$ and $u$, not in mean square. 
\end{remark}

In the next step we deal with the pathiwise uniqueness. In addition to Assumption \ref{noise1} we assume the following.

\begin{enumerate}[start=2,label=\textbf{(A\arabic*)}]
	\itemsep0.3em
	\item \label{noise2} 
The operator $G$ is Lipschitz continuous, i.e.,
\begin{equation*} 
  \exists \ L_G >0 : \quad  \|G(u)-G(v)\|_{L_{HS}(U,H)} \le L_G\|u-v\|_H \qquad \forall \ u, v \in H.
\end{equation*}
\end{enumerate}

Then we consider the pathwise uniqueness. When the noise is additive, the difference of two solutions fulfills an equation without noise term and the proof of uniqueness resembles that in the deterministic setting (see, e.g., \cite{Fer03}). Some changes are required when the noise depends on the unknown velocity $u$ (see \cite{FerZanSNS} and references therein).  We have the following result, based on a technique by Schmalfuss \cite{Sch1997}.
\begin{proposition}
\label{prop_uniq_sol}
Let $T>0$. To the assumptions of Proposition \ref{teo-esistenza-sol-mart} we add \ref{noise2}.
\\
Let $(\widetilde\Omega, \widetilde{\mathcal{F}}, \widetilde{\mathbb{P}},  \widetilde{\mathbb{F}},u_i)$, $i=1,2$ be two martingale solutions to \eqref{sns} with the same initial velocity. Then $\widetilde{\mathbb{P}}(u_1(t)=u_2(t)\, \ \text{for all} \ t \in [0,T] )=1$, that is solutions to equation \eqref{sns} are pathwise unique.
\end{proposition}

\begin{remark}\label{oss-regol}
A more regular solution can be obtained in a similar way, if the initial velocity and the forcing terms are more regular in space. For example one considers the dynamics not in the space $H$ but in a more regular space ${\cal D}(A^a)$ for  some $a>0$. We refer to \cite{Fla94} and \cite{Fer97}
in the case of additive noise.

By the way, notice that the papers on the bidimensional Navier-Stokes equations which exploit the conservation of the enstrophy (which is the curl of the velocity) consider the problem in a box with periodic boundary conditions, i.e. they work on the torus. In that setting regularity results can be easily obtained in any space ${\cal D}(A^k)$, i.e.  for any $k>0$.
\end{remark}

\section{Existence of invariant measures}
\label{S:ex-inv-meas}

Starting from  this section we assume that $f$ is stationary in time,
i.e. it is independent of time 

As soon as the existence of martingale solutions and the pathwise uniqueness are obtained,  the solution is a strong one  in the sense of Definition \ref{strong_sol_def}. 
This is the starting point to define the Markov semigroup. For the  general definitions of this section  we refer to the books \cite{DapZab1,DapZab2} by Da Prato and Zabczyk. 

Let us fix some notations.  Given a Hilbert space $E$, 
we denote by 
 $\mathcal{B}_b(E)$ the  space of all bounded Borel functions   $\phi: E \to \mathbb R$, and by 
$\mathcal{C}_b(E)$
the space of all continuous and bounded functions $\phi:E\to\mathbb R$. 
Moreover we say that a sequence $\{\mu_n\}_n\subset {\cal P}(E)$ of probability measures  on the Borel subset of $E$ weakly converges to a  
probability measure $\mu\in {\cal P}(E)$ when
\[
\int_E \phi \, {\rm d}\mu_k \to \int_E \phi\,{\rm d}\mu \quad \text{as} \ k \rightarrow \infty, \quad \forall \ \phi \in \mathcal{C}_b(E).
\]
For short we write $\mu_k  \rightharpoonup\mu$.

We denote by $u(t;x)$ the solution  evaluated at time $t>0$,  started at time $0$  from $x$.
We know  that for any $x\in H$ there exists a unique strong solution $u(t;x)$ defined for any time $t\ge0$. 
We introduce the operator $P_t$  for each fixed time  $t>0$:
\begin{equation}
    \label{sem}
(P_t\phi)(x)=\mathbb E(\phi(u(t;x))), \qquad  \phi\in \mathcal{B}_b(H)
\end{equation}
This defines a  Markov semigroup acting on $\mathcal{B}_b(H)$. 
When we consider more regular solutions as in Remark \ref{oss-regol}, we consider 
it acting on functions $\phi\in \mathcal{B}_b(E)$ for some more regular Hilbert space $E\subset H$. Let us introduce the definitions in this more general setting.

We consider $P_t:\mathcal{B}_b(E)\to \mathcal{B}_b(E)$.
Its dual operator defines the semigroup 
$\{P_t^{\star}\}_{t \ge 0}$ which governs the evolution of the laws.
We say that a  probability measure $\mu$, defined on the Borel subsets of $E$, is {\em invariant} if
\[P_t^{\star} \mu = \mu \qquad \forall t\ge 0.\]
Equivalently, for arbitrary $t$,
\[
\int_E P_t \phi \: d\mu =\int_E \phi \:d\mu \hspace{1cm} \forall \phi \in \mathcal{B}_b(E).
\]

The first property  one looks for is  the Feller property of  the Markov semigroup, i.e.  
$P_t:\mathcal{C}_b(E)\to \mathcal{C}_b(E)$ for any $t\ge0$.
We always assume that it holds true. We will provide details in the proof of the next Theorem.

As explained in the Introduction, we are interested in the long time behaviour 
of the system \eqref{sns}.
The invariant measures are strictly related to the asymptotic behaviour of the law of the 
solution process $u$.
As a first result we have that if for some initial condition $x$ the law of $u(t;x)$ 
weakly converges to some probability measure $\mu$ as $t\to+\infty$, then 
$\mu$ is an invariant measure.

Weaker conditions allow to get the existence of an invariant measure. 
Usually one proves the existence of an invariant measure by means of the 
Krylov-Bogoliubov's technique. We state it in the next Proposition. It involves the time averages of the laws
 instead of asking directly the (stronger) convergence of the law, as explained before. 
 Let us denote by ${\cal L}(u(t;x))$ the law of  $u(t;x)$ for fixed time $t>0$ and initial velocity $x$.

\begin{proposition} \label{kryl-bog}
Assume that the Markov transition semigroup  is Feller in $E$, i.e. 
$P_t:\mathcal{C}_b(E)\to \mathcal{C}_b(E)$ for any $t$.
\\
Suppose that there exists an initial condition $x\in E$ and  a sequence of times   $t_n \uparrow +\infty$ such that
\[
{1 \over t_n} \int_0^{t_n} {\cal L}(u(r;x)) dr \rightharpoonup \mu
\hspace{1cm} \mbox{  as } n \rightarrow +\infty.
\]
Then $\mu$ is an invariant measure.
\end{proposition}

We now quote  the two first results \cite{Fla94} and \cite{chow} for the bidimensional Navier-Stokes equation \eqref{sns} with additive noise operator $G$. Many others have been obtained later on.
\begin{theorem}
 Let $f$ and $G$ satisfy one of the following conditions:
 \begin{itemize}
 \item[i)] the range of $G$ is contained in ${\cal D}(A^{\frac 14 +\beta_0})$ for some positive  $\beta_0$ and 
$f\in {\cal D}(A^{-\frac 12+\theta})$ 
for some $\theta\in (0,2\beta_0)\cap (0,\frac 12]$;
\item[ii)] the range of $G$ is contained in ${\cal D}(A^{\frac 12 +\beta_0})$ for some positive  $\beta_0$ and 
 $f\in V^\prime$.
 \end{itemize}
 Then the Navier-Stokes equation \eqref{sns} has an invariant measure.
\end{theorem}
\begin{proof}
For the Navier-Stokes equation \eqref{sns}, the Feller property in $H$ is easily obtained; roughly speaking it corresponds to the property of continuity of the solution from the initial data.
The convergence of the time averages as in Proposition \ref{kryl-bog} 
is obtained by means of a tightness of the laws, and then using the Prokhorov theorem. 

In case i), for initial velocity $x=0$,  in \cite{Fla94}  a  uniform (in time) bound in probability is proved; this depends on the 
dissipative property of the Navier-Stokes equation. 
For any $\epsilon>0$ there exists 
$M_\epsilon>0$ such that
\[
\sup_{r>0} P( \|A^{\min(\frac 14+\beta,\theta)} u(r;0)\|_H>M_\epsilon)<\epsilon,
\]
for any $\beta<\beta_0$. 
\\
The tightness follows from the compact embedding
${\cal D}(A^{\min(\frac 14+\beta,\theta)})\subset H$.

In case ii),   the mean square estimate \eqref{mean-sq-energy} is used  to get the tightness of the laws 
\[ \frac 1t \int_0^t \mathbb E\|u(r;0)\|_V^2{\rm d}r\le 
 \frac 1{\nu^2} \|f\|^2_{V^\prime}+ \frac 1\nu C_2,
 \qquad \forall t>0 
\]
where for simplicity we consider again initial velocity $x=0$.
We write it here with $f$ even if the deterministic forcing term is not considered in \cite{chow}. The tightness in $H$ follows from the compact embedding $V\subset H$. The latter estimate provides also that the invariant measure has support contained in the space $V$.
\end{proof}

In conclusion, the assumptions to get the existence of at least one  invariant measure concern some space regularity of the forcing terms.
Even $G=0$ can be considered; in this case any stationary solution of the deterministic Navier-Stokes equation can be viewed as an invariant measure.

\begin{remark}
\label{star}
In a similar way,
the existence of an invariant measure is obtained in the multiplicative case when the operator $G$ fulfils \ref{noise1}, \ref{noise2} and $C_1<4\nu\lambda_1$. 
We refer to   \cite{FlaGat95} and  \cite{FerZanSNS}. 
\end{remark}

\section{Uniqueness of the invariant measure}
\label{S:unique-meas}

Proving the uniqueness of the invariant measure is in general more challenging than showing its existence. 
Such a problem has been investigated mostly for the case of an additive noise, whereas the 
results for the case of a multiplicative noise are much scarcer. 
To start let us remind the results for the deterministic Navier-Stokes equation, i.e. equation \eqref{sns} with  $G=0$ (see \cite[Theorem II.1.3]{Temam2001}).
\begin{theorem}
    For every $f\in V^\prime$ and $\nu>0$,  there exists at least one solution of the steady Navier-Stokes equation
    \[
-  \nu \Delta u +(u \cdot \nabla)u +\nabla p =f ; 
\qquad \mbox{\text{div} }u =0 
    \]
If
\[\nu^2>C_1 \|f\|_{V^\prime}\]
where $C_1$ is a constant depending only on the domain, then the stationary solution is
unique.
\end{theorem}
There are some specific results that assert non-uniqueness when the viscosity is not large compared with the external force
 (see, e.g., Section II.4 of \cite{Temam2001}).

When $G\neq 0$, there are many results for the   uniqueness of the invariant measure. 
They involve qualitative conditions on $f$ and $G$;  sometimes even  quantitative conditions   depending on the viscosity $\nu$.
These are based on different techniques. Here we classify them in terms of the  amount of degeneracy of the stochastic forcing term. 
We do not claim to be exhaustive.
This is not meant to be an all encompassing review article as we said in the Introduction. 

The starting point is the dissipative property of the Navier-Stokes equations, due to the term $-\nu\Delta u$ in equations \eqref{eq-fluid}. Considering $u(t)=\sum_{n=1}^\infty u_n(t)e_n$, the evolution of each component $u_n(t)=(u(t),e_n)_H$ is
\begin{equation}
    du_n(t)+[\nu \lambda_n u_n(t)+ \langle B(u(t),u(t)),e_n\rangle] \; {\rm d}t=\langle f, e_n\rangle  {\rm d}t + \langle G(u(t)) {\rm d}W(t), e_n\rangle
\end{equation}
Since $\lambda_n \to +\infty$ as $n\to +\infty$, we have that for $n$ large enough the dissipation is strong and,  roughly speaking, dominates the nonlinearity; however, for small values of $n$ the deterministic nonlinear dynamics is unstable and the role of the noise is crucial to get a unique asymptotic behavior. 

Technically speaking, when the forcing acts on all modes we are in the  “elliptic setting", whereas  when it acts only on the unstable low modes   we  are in the “effectively elliptic setting". 
In both the above scenario, since the noise affects all of the unstable directions, it is not surprising that the Navier-Stokes equation has no more than one invariant measure. 
Anyway the techniques to prove the uniqueness of the invariant measures are different in these two settings.
Lastly, there is the  “hypoelliptic setting" when  one allows for unstable directions not to be directly forced by the noise. The uniqueness of the invariant measure is a consequence of the fact that  the randomness can reach all of the unstable modes;  in fact, the noise is transmitted to the relevant degrees of freedom through the drift term $B$. 

In the next subsections we consider in detail these settings.
Under  assumptions \ref{noise1}-\ref{noise2}, the only work that proves the uniqueness of the invariant measure is \cite{FerZanSNS}. This is the subject of next  Subsection \ref{effect}. 
Then, we revise the main uniqueness results in the case of an additive noise, full or not (i.e. acting on  all modes or on a finite number of modes).

\subsection{Effectively elliptic setting}
\label{effect}

A powerful and flexible technique to investigate the ergodic properties of a system is given in \cite{GHMR17}. 
Here Glatt-Holtz, Mattingly and Richards  identify an intuitive and conceptually simple framework for proving the uniqueness of the invariant measure by a \textit{generalized asymptotic coupling technique} in the “effectively elliptic” setting, that is when the range of the  operator $G$ contains  all the unstable directions.
We speak of “essentially elliptic” setting, since the directions essential to determine the system’s long time
behavior, the unstable directions, are directly forced. 

The \textit{generalized asymptotic coupling technique} 
allows  to work with noises that can be degenerate in an infinite number of directions. 
We shortly sketch the idea behind this method.
The system has finitely many unstable directions (low modes)
and infinitely many stable directions (high modes). One can then use the noise to steer the
unstable directions together and let the dynamics cause the stable directions
to contract, in the spirit of \cite{FP}. Since these techniques used Girsanov’s theorem on some finite dimensional dynamics, they required
that all of the unstable directions are directly forced; this is in fact a type of partial ellipticity assumption. 

Let us briefly recall the results in \cite{GHMR17} in the best form that fits our framework, thus not considering the most possible general case as in \cite{GHMR17}. We denote by $H^{\mathbb{N}}$, with its Borel $\sigma$-field $\mathfrak{B}(H^{\mathbb{N}})$, the space of one-sided infinite sequences $H^{\mathbb{N}}$. By $\mathcal{P}\left(H^{\mathbb{N}}\right)$ we denote the collections of Borel probability measures on $(H^{\mathbb{N}}, \mathfrak{B}(H^{\mathbb{N}}))$.
For given $\mu, \nu \in \mathcal{P}(H^{\mathbb{N}})$ we define the space
\begin{equation*}
\widehat{\mathcal{C}}
(\mu, \nu):= \{ \xi \in \mathcal{P}(H^{\mathbb{N}} \times H^{\mathbb{N}})\ : \ \pi_1(\xi) \ll \mu, \ \pi_2(\xi) \ll \nu\}
\end{equation*}
and call any probability measure from the class $\widehat{C}(\mu, \nu)$ a \textit{generalized coupling} for $\mu, \nu$.
Here $\pi_i(\xi)$ denotes the $i$-th marginal distribution of $\xi$, $i=1,2$ and we recall that $\mu \ll \nu$ means that $\mu$ is absolutely continuous w.r.t. $\nu$.
Let $\{u(n)\}_{n \in \mathbb{N}}$ be a Markov chain with state space $(H, \mathfrak{B}(H))$. The law of the sequence $\{u(n)\}_{n \in \mathbb{N}}$ in $(H^{\mathbb{N}}, \mathfrak{B}(H^{\mathbb{N}}))$ with initial datum $x \in H$ is denoted by $\mathbb{P}_{x} \in \mathcal{P}\left(H^{\mathbb{N}}\right)$. 

We have the following result.
\begin{theorem}
    \label{GHMRthm}
If for any $x, \tilde x \in H$ there exists some $\xi:=\xi_{x, \tilde x} \in \widehat{\mathcal{C}}(\mathbb{P}_{x}, \mathbb{P}_{\tilde x})$ with 
\begin{equation}
\label{cond_GHMR}
\xi \left((u, \tilde u) \in H^\mathbb{N} \times H^{\mathbb{N}} \ : \ 
 \lim_{n \rightarrow \infty} \|u(n) - \tilde u(n)\|_H=0
\right)>0,
\end{equation}
then there exists at most one ergodic invariant probability measure $\mu$ supported on $H$.  
\end{theorem}

Theorem \ref{GHMRthm} have been used in \cite{GHMR17} to infer the uniqueness of the invariant measure for system \eqref{sns} driven by ad additive noise.
\begin{theorem}
\label{GHMinv}
Let $G \in L_{HS}(U,H)$ in \eqref{sns} be a constant operator.
For every $\nu>0$ there exists $\bar N=\bar N(\nu, \|f\|_{V'}, \|G\|_{L_{HS}(U,H)})$ such that, if 
\begin{equation}
\label{nondeg}
\emph{Rg}(G) \supseteq P_N(H),
\end{equation}
for some $N \ge \bar N$, 
then \eqref{sns} admits a unique invariant measure $\mu\in \mathcal{P}(H)$.
\end{theorem}

An example of stochastic forcing term satisfying the above requirements is given by  $G{\rm d}W(t)=\sum_n G f_n {\rm d}\beta_n(t)$ (see \eqref{def-serie-W}) with 
\[
Gf_n=\begin{cases}  \sigma_n\in H,& n\le N,\\0, & n>N. \end{cases}
\]
This noise is finite dimensional, that is only a (sufficiently large) number of modes are stochastically forced.
These are the determining modes (low modes). 
The proof of Theorem \ref{GHMinv} relies on Theorem \ref{GHMRthm}: the idea is to introduce a modification of the Navier-Stokes equation \eqref{sns} such that: 
\begin{enumerate}[label=\textbf{(\arabic*)}]
	\item \label{i} 
 the law of the solution to the new SPDE is absolutely continuous with respect to the law of the solution to the original one \eqref{sns}; 
 \item \label{ii}  for any pair of distinct initial conditions, there is a positive probability that solutions to these systems asymptotically  converge, when evaluated  on a infinite sequence of evenly spaced times. 
 \end{enumerate}
 Such a modification is obtained by introducing a suitable (finite dimensional) shift in the driving Wiener process to force solutions, which start at different initial conditions, together asymptotically as time goes to infinity. The key role in the proof is played by some \textit{pathwise} Foias-Prodi estimates that quantify the minimum number $\bar N$ of modes that need to be activated by the noise in order to get synchronization at infinity. 

In \cite{FerZanSNS} we showed that the generalized coupling techniques of \cite{GHMR21} are flexible enough to deal with noises of multiplicative type. 
In addition to Assumptions \ref{noise1}-\ref{noise2} we impose the following condition on the operator $G$.
\begin{enumerate}[start=3,label=\textbf{(A\arabic*)}]
	\itemsep0.3em
	\item \label{noise3} 
There exists a measurable map $g:H \rightarrow L(H,U)$  such that
\begin{equation}
\label{bound_g}
\sup_{u \in H}\|g(u)\|_{L(H,U)} < \infty
\end{equation}
and
\begin{equation}
\label{GgM}
G(u)g(u)=P_M \qquad \forall\  u \in H,
\end{equation}
for a positive integer $M$.
\end{enumerate}
Notice that the existence of a map $g: H \rightarrow L(H,U)$  fulfilling \eqref{GgM} is equivalent to the following property 
\begin{equation*}
\text{{Rg}}G(u) \supseteq P_MH\qquad \forall\  u \in H,
\end{equation*}
that is Assumption \ref{noise3} can be seen as a non degeneracy condition on the low modes.
An example of stochastic forcing term satisfying the above requirements is given by
\[
G(u){\rm d}W(t)=\sum_{n=1}^M G(u)[f_n]{\rm d}\beta_n(t), \qquad  \text{with} \quad G(u)[f_n]= \frac{\sqrt{\|u\|^2_H +1}}{n+1} e_n, \qquad u \in H.
\]
Our main result reads as follows.
\begin{theorem}
\label{unique_thm}
Let Assumptions \ref{noise1}-\ref{noise2} be in force. Moreover, assume that $\nu >  \frac{11C_1}{4\lambda_1}$.
Then there exists  a positive integer
\begin{equation*}
\bar N=\bar N(L_G, C_1, C_2, \nu, \lambda_1, \|f\|_{V'}),
\end{equation*}
 such that, if \ref{noise3} holds for some  $M \ge \bar N$, then \eqref{sns} possesses at most one ergodic invariant measure $\mu\in \mathcal{P}(H)$. 
\end{theorem}
As far as we know the above result is the only one in the literature providing the uniqueness of the invariant measure for equation \eqref{sns} driven by a multiplicative noise that satisfies a linear growth condition (see Assumption \ref{noise2}). 
When instead the noise is bounded, apart from our results in  \cite{FerZanSNS}, there are results in the papers \cite{Oda2008} and \cite{Mat02b}, where uniqueness of the invariant measure is proved by means of \textit{coupling techniques}, still in the effectively elliptic setting.

We briefly sketch the proof of Theorem \ref{unique_thm}. We appeal to Theorem \ref{GHMRthm}: the idea is to introduce a suitable modification of the Navier-Stokes equation \eqref{sns} that satisfies conditions \ref{i}-\ref{ii} above. 
\\
Given $u_0 \in H$, let $u=u(u_0)$ denote the  solution of the Navier-Stokes equation \eqref{sns}.
Given $N>0$ and $v_0\in H$, let $v=v(v_0,u_0)$ denote the solution to the following nudged equation
\begin{equation}
\label{snsnud}
\begin{cases}
{\rm d}v(t) + \left[\nu Av(t)+B(v(t),v(t))\right]\,{\rm d}t
= 
f \,{\rm d}t +G(v(t))\,{\rm d}W(t)+ \frac{\nu\lambda_N}{2} P_N(u(t)-v(t))\,{\rm d}t, \qquad t>0
\\
v(0)=v_0.
\end{cases}
\end{equation}
The effect of the nudging term $\frac{\nu\lambda_n}{2} P_N(u-v)$ is to drive $v$ towards $u$ on the finite dimensional manifold $P_NH$, that is on the low modes. We derive some Foias-Prodi estimates \textit{in expectation} that in fact quantify how many modes $N$ need to be activated in order to synchronize the full solution. \begin{lemma}[Foias-Prodi estimates in expected value.]
\label{FPlem}
Provided $\nu > \frac{3C_1}{4\lambda_1}$ there exists a positive integer $\bar N=\bar N(L_G, C_1, C_2, \nu, \lambda_1, \|f\|_{V'})$ and a positive constant $C$ depending on the structural parameters of the equations and the initial data $u_0$, $v_0$, such that for any  $N \ge \bar N$, where $N$ is the parameter appearing in equation \eqref{snsnud}, the estimate
\begin{equation}
\label{FP}
\mathbb{E}\left[ \|u(t)-v(t)\|^2_H \right] 
\le\frac{C}{t^p}\qquad  t>0,
\end{equation}
holds for any 
$p\in \left(0,\frac{\nu \lambda_1}{ 2 C_1}-\frac 38\right)$. 
\end{lemma}
These estimates show that a finite dimensional control, when chosen in a proper way, allows to synchronize (in the mean)  any two solutions in the limit as $t \rightarrow +\infty$. 
\begin{remark}
We emphasize that if Assumption \ref{noise1} holds with $C_1=0$, i.e. the noise is bounded, then the Foias-Prodi estimate holds for every $\nu>0$ and is of the form 
\begin{equation}
\label{FP}
\mathbb{E}\left[ \|u(t)-v(t)\|^2_H \right] 
\le C e^{-\delta t},\qquad  t>0,
\end{equation}
where $C$ is a positive constant depending on the structural parameters of the equations and the initial data $u_0$, $v_0$, 
and $\delta$ is a positive constant depending on the structural parameters of the equations.
\\
The multiplicative bounded noise is the one that most closely resembles the additive noise case, since one obtains an exponential decay. However, when the noise is additive, one can work pathwise and obtain \emph{pathwise} Foias-Prodi estimates, that is there exists a positive integer $\bar N=\bar N(\|G\|^2_{L_{HS}(U,H)}, \nu, \lambda_1, \|f\|_{V'})$ such that for any $N \ge \bar N$,
\[
\lim_{t \rightarrow +\infty}\|u(t)-v(t)\|_H^2=0 \qquad \mathbb{P}-a.s.,
\]
exponentially fast. These are the estimates considered in \cite[Section 3.1]{GHMR17}.
\end{remark}
In view of Lemma \ref{FPlem} from now on we assume that the parameter $N$ appearing in equation \eqref{snsnud} is such that $N \ge \bar N$. At this point we notice that the nudged equation \eqref{snsnud} can be equivalently rewritten as
\begin{equation}
\label{snsnud2}
\begin{cases}
{\rm d}v(t) + \left[\nu Av(t)+B(v(t),v(t))\right]\,{\rm d}t
= 
f \,{\rm d}t +G(v(t))\,{\rm d}\widetilde W(t), \qquad t>0
\\
v(0)=v_0,
\end{cases}
\end{equation}
where 
\begin{equation*}
\widetilde{W}(t):=W(t)+ \int_0^t h(s)\, {\rm d}s,
\end{equation*}
and 
\begin{equation}
\label{sigma}
h(t):= \frac{\nu \lambda_N}{2}\ g(v(t))\ P_N(u(t)-v(t)), \qquad t \ge 0.
\end{equation}
Thanks to the Girsanov Theorem (whose assumptions are verified exploiting the Foias-Prodi estimates \eqref{FP}) one can show that, if assumption \ref{noise3} holds with $M \ge N$, then  the law of the solution to system \eqref{snsnud2} is absolutely continuous w.r.t. the law of the solution to equation \eqref{sns} starting form $v_0 \in H$, as measures on $C([0, + \infty);H)$. We have thus constructed a modification of the original equation \eqref{sns} satisfying condition \ref{i} above. It remains to verify that condition \ref{ii} holds true, that is 
\[
\mathbb{P} \left(\lim_{n \rightarrow \infty} \|u(n)-v(n)\|_H=0 \right) >0.
\]
We introduce the event 
\begin{equation*}
B:= \bigcap_{m=1}^{\infty} \bigcup_{n=m}^{\infty} B_n,
\qquad \text{where} \qquad 
B_n:= \left\{ \|v(n)-u(n)\|^2_H > \frac{1}{n^2}\right\}.
\end{equation*}
Exploiting again the Foias-Prodi estimates \eqref{FP} one can show that 
\[
\mathbb{P}(B) \le 
\dfrac{C}{R^p} \qquad   
\text{for} \quad p\in\left(0, \frac{\nu \lambda_1}{2 C_1}- \frac 38\right), 	
\]
where the constant $C$ do not depend on $R$.
By choosing $R$ sufficiently large, we have that $\mathbb{P}(B)$ is close to $0$, hence $\mathbb{P}(B^c)$ is close to $1$.
Thus, from the continuity from below, we can find $m^*>0$ sufficiently large so that 
\begin{equation*}
\mathbb{P}\left( \bigcap_{n=m^*}^{\infty}B_n^c \right)>\frac12. 
\end{equation*}
Now we observe that 
\[
\left\{ \lim_{n \rightarrow 0} \|u(n)- v(n)\|^2_H =0 \right\} \supseteq \bigcap_{n=m^*}^{+\infty}B_n^c,
\]
hence,
\begin{equation}
\label{lim}
\mathbb{P} \left( \lim_{n \rightarrow 0} \|u(n)- v(n)\|^2_H =0 \right) \ge \mathbb{P} \left(  \bigcap_{n=m^*}^{+\infty}B_n^c\right)> \frac 12 >0.
\end{equation}
This concludes the proof of the uniqueness of the invariant measure for equation \eqref{sns}. 
In fact, coming back to the abstract Theorem \ref{GHMRthm}, for any $u_0, v_0 \in H$ we have constructed the measure $\xi_{u_0,v_0}$ on $H^{\mathbb{N}} \times H^{\mathbb{N}}$ given by the law of $(u(n), v(n))_{n \in \mathbb{N}}$, where $u$ solves \eqref{sns} with corresponding initial datum $u_0$ and $v$ solves \eqref{snsnud} with corresponding initial datum $v_0$, where we consider $N \ge \bar N$ with $\bar N$ as in Lemma \ref{FPlem}. Provided Assumption \ref{noise3} holds with $M \ge N$, $\pi_2(\xi_{u_0,v_0}) \ll \mathbb{P}_{v_0}$. We therefore have that $\xi_{u_0, v_0} \in \widehat{\mathcal{C}}\left(\mathbb{P}_{u_0}, \mathbb{P}_{v_0}\right)$ and, in view of \eqref{lim}, it holds
\begin{equation*}
\xi_{u_0, v_0} \left((u, v) \in H^\mathbb{N} \times H^{\mathbb{N}} \ : \  \lim_{n \rightarrow \infty} \|u(n) - v(n)\|_H=0\right)>0.
\end{equation*}

\begin{remark}
In \cite{FerZanSNS} we prove an analogous version of Theorem \ref{unique_thm} in the case of a noise which is bounded. 
In this case the result holds for any $\nu>0$.
\end{remark}

\begin{remark}
In Theorem \ref{unique_thm}  there is the condition $\nu >  \frac{11C_1}{4\lambda_1}$: the  viscosity coefficient $\nu$ has to balance the intensity of the \emph{multiplicative part} of the noise. 
This condition should not be surprising (see \cite{noi} for a similar situation); a similar condition, although weaker, appears also to get the existence of an invariant measure (see Remark \ref{star}).
\\
Actually, condition $\nu >  \frac{11C_1}{4\lambda_1}$ can be weaken to $\nu >  \frac{3C_1}{4\lambda_1}$ introducing a suitable localization term; we refer to \cite[Section 6]{FerZanSNS} for more details. 
We skip the details here to be as concise as possible. On the other hand, the introduction of the localization term is entirely superfluous when considering an additive or a bounded multiplicative noise. This has to do with the faster decay in the Foias-Prodi estimates: exponential versus polynomial.
\end{remark}

To conclude this part, we recall that in the “effectively elliptic setting", the results of unique ergodicity can also be proved by means of 
different 
techniques. We mention \cite{EMS01}, \cite{KS00}, \cite{BKL01}, \cite{KS01}, \cite{Mat02b}, \cite{Hai02} for the case of an additive noise and \cite{Oda2008}, \cite{Mat02b} 
for the case of a bounded multiplicative noise. For further references see \cite{Mat04} and \cite{CIME}.


\subsection{Elliptic setting}

In the elliptic setting sufficient conditions to ensure the uniqueness of an invariant measure are the strong Feller property and the irreducibility for a Markov transition semigroup as defined in \eqref{sem}.

\begin{definition}
A Markov transition semigroup $\{P_t\}_{t \ge 0}$ on a Hilbert space 
$E$ is said
to be 
\begin{itemize}
    \item 
\emph{strong Feller} if  $P_t:{\cal B}_b(E)\to {\cal C}_b(E)$  for any $t>0$;
    \item
    \emph{irreducible} if for all $x, y \in E$, $r>0$ and $t>0$ it holds 
\[
P_t \pmb{1}_{B_y(r)}(x)>0, 
 \qquad \text{with} \quad B_y(r):=\{ x \in E \ : \ \|x-y\|_E < r\}.
\]
\end{itemize}
\end{definition}
The latter property can be written also as 
\[
\mathbb{P}\left( u(t;x) \in B(y,r)\right)>0
\]
for all $x, y \in E$, $r>0$ and $t>0$. 
This means that  starting from any point in $E$ the process visits  immediately  any non-empty open subset of $E$. We point out that the irreducibility does not hold for the deterministic Navier-Stokes equation; indeed,  if $G \equiv 0$ then ${\cal L}(u(t;x))= \delta_{u(t;x)}$.

A standard criterion for the  strong Feller property is given by \cite[Lemma 7.1.5]{DapZab2}
\begin{proposition}
A semigroup $\{P_t\}_{t \ge 0}$ on a Hilbert space $E$ is strong Feller if, for all $\varphi:E \rightarrow \mathbb{R}$ with $\|\varphi\|_\infty:=\sup_{x \in E}|\varphi(x)|$ and $\|\nabla \varphi\|_\infty$ finite, one has
\begin{equation*}
    \|\nabla P_t \varphi(x)\|\le C(\|x\|)\|\varphi\|_\infty,
\end{equation*}
for all $x \in E$, where $C : \mathbb{R}^+\rightarrow \mathbb{R}$ is a fixed nondecreasing function.
\end{proposition}

When irreducibility and strong Feller property hold, unique ergodicity follows from the following theorem due to Khasminskii and Doob (see \cite{DapZab2}). 
\begin{theorem}
\label{ergod}
Assume that a stochastically continuous Markovian transition semigroup $\{P_t\}_{t \ge 0}$ on a 
Hilbert space $E$ is strong Feller and irreducible. Then, 
there exists at most one invariant measure $\mu$. Moreover, if such a measure exists, 
\begin{itemize}
\item it is  strongly mixing, that is 
\[
\lim_{t\to+\infty} P(u(t;x)\in\Gamma)=\mu(\Gamma)
\]
for all Borel subsets $\Gamma\subset E$ and $x\in E$;
\item 
all the laws ${\cal L}(u(t;x))$ are equivalent to $\mu$, for any $x\in E$ and $t>0$.
\end{itemize}
\end{theorem}
It follows that $\mu$ is also ergodic in the sense that
\[
\lim_{t \rightarrow + \infty} {1 \over t} \int_{0}^{t}
        \varphi(u(s;x)) \; {\rm d}s = \int_{E} \varphi \; {\rm d}\mu    \hspace{1cm}
        \mathbb P-\mbox{a.s.} 
\]
for all $x \in E$ and $\varphi \in \mathcal{B}(E)$, $\mu$-integrable.


Theorem \ref{ergod} has been successfully used in \cite{FlaMas95} when the noise driving equation \eqref{sns} is additive.
There the authors prove the following result for the Navier-Stokes equation \eqref{sns}.
\begin{theorem}
Assume  $f \in H$ and $G:H \rightarrow H$ is a bounded injective  linear  operator such that its  range $\emph{Rg}(G)$  fulfills
\begin{equation}
\label{non_deg_FM}
\mathcal{D}(A^{\frac12}) \subseteq \emph{Rg}(G) \subseteq \mathcal{D}(A^{\frac 38 + \varepsilon})
\end{equation}
for some $\varepsilon>0$. 
Then \eqref{sns} possesses only  one invariant measure, which is ergodic and strongly mixing.
\end{theorem}

Irreducibility is proven thanks to the assumption Rg$(G) \subset \mathcal{D}(A^{\frac 38 + \varepsilon})$. 
The strong Feller property is obtained by means of the Bismut-Elworthy-Li formula applied to a finite dimensional truncated approximation of the Navier-Stokes equations \eqref{sns}; it requires the lower bound
$\mathcal{D}(A^{\frac12}) \subset \text{Rg}(G)$, which roughly speaking corresponds to an upper bound for the inverse operator $G^{-1}$.

We point out that the conditions on the operator $G$ impose that $G$ is a full operator, i.e. the noise  acts on all the components. 
And, roughly speaking,  $G$ has to be enough regular but not too much.
The conditions on  the forcing terms  do not depend on the viscosity coefficient $\nu$. This is in line with the observation that, since the noise affects all the modes,  even a noise of small intensity allows to get a unique invariant measure.

An example of stochastic forcing term satisfying the above conditions is $G=A^{-a}$ with $\frac 38<a\le \frac 12$, when $U=H$.

The results of \cite{FlaMas95} have been generalized in \cite{Fer97}, \cite{Fer99}  removing the bound ${\cal D}(A^{\frac 12})\subseteq \text{Rg}(G)$. 
Recently \cite{ferrari} extended these results so to remove the other bound. The technique introduced by Ferrari is based on the irreducibility property and a modified strong Feller property, proving the uniqueness of the invariant measure when ${\cal D}(A^{\frac 12})\subseteq \text{Rg}(G)\subseteq {\cal D}(A^{\frac 14+\epsilon})$. 
Anyway the limit case of a space-time white noise, i.e. $G=Id$, cannot be reached by means of  these techniques.

\begin{remark}
We are not aware of results that prove  the uniqueness of the invariant measure in the case of a multiplicative noise by means of Theorem \ref{ergod}. 
\end{remark}
\begin{remark}
    When $G$ is a full noise, the linear Stokes equation \eqref{OU-eq} has a unique invariant measure which is normally distributed and has full support in the sense that it is not concentrated on finite dimensional subsets of $H$ (whereas  this happens in the case of degenerate noise). It would be interesting to compare the two unique invariant measures for the stochastic Stokes and the stochastic Navier-Stokes equations, respectively.
    So far results have been obtained for the  hyperviscous stochastic Navier-Stokes equation (see Section 4 in \cite{Fer-COSA}).
\end{remark}


\subsection{Hypoelliptic setting}


In \cite{HaiMat}, Hairer and Mattingly observed that the strong Feller property is neither essential nor natural for the study of ergodicity in dissipative infinite-dimensional systems. 
They introduced a weaker asymptotic strong Feller property, which is sufficient to give unique ergodicity in the “hypoelliptic setting", i.e. when  some modes are excited by noise and the nonlinear term propagates the noise to the whole system. 
The action of the bilinear term is to mix up the dynamics among different Fourier components; 
precise computations have been given when working in a bidimensional torus, because in that setting there is an explicit expression of the Fourier components of the bilinear operator $B$ and of the  the eigenvalues and eigenfunctions of the Stokes operator $A$ in the equation \eqref{sns}.

For the precise definition of asymptotic Strong Feller we refer the reader to \cite{HaiMat}. Here we recall the following characterization (see \cite[Proposition 3.12]{HaiMat}).
\begin{proposition}
Let $t_n$ and $\delta_n$ be two positive sequences with $\{t_n\}_n$ non-decreasing and $\{\delta_n\}_n$ converging to zero. A semigroup $\{P_t\}_{t \ge 0}$ on a Hilbert space $E$ is \emph{asymptotically strong Feller} if, for all $\varphi: E \rightarrow \mathbb{R}$ with $\|\varphi\|_\infty$ and $\|\nabla \varphi\|_{\infty}$ finite, one has
\begin{equation*}
    \|\nabla P_{t_n} \varphi(x)\|\le C(\|x\|)\left(\|\varphi\|_\infty + \delta_n\|\nabla \varphi\|_\infty\right),
\end{equation*}
for all $n$ and $x \in E$, where $C : \mathbb{R}^+\rightarrow \mathbb{R}$ is a fixed nondecreasing function.
\end{proposition}
One then introduces the following form of topological irreducibility.
\begin{definition}
We say that a Markov semigroup $\{P_t\}_{t \ge 0}$ is \emph{weakly topologically irreducible} if for all $x, y \in E$ there exists a $z \in E$ so that, for any open set $A$ containing $z$, there exists $s,t>0$ with $P_s \pmb{1}_A(x)>0$ and $P_t \pmb{1}_A(y)>0$. 
\end{definition} 

When weak irreducibility and asymptotic strong Feller property hold, unique ergodicity follows from the following
theorem; see \cite[Corollary 1.4]{HaiMat2}.
\begin{theorem}
\label{HaiMat}
Any Markov semigroup $\{P_t\}_{t \ge 0}$ on 
a Hilbert space which is Feller, weakly topologically irreducible and asymptotically strong Feller admits at most one invariant probability measure.
\end{theorem}

Theorem \ref{HaiMat} has been used in \cite{HaiMat} (see Theorem 2.1) to infer the uniqueness of the invariant measure in the following setting. The authors work on the two dimensional torus $\mathbb{T}^2=[-\pi, \pi]^2$ without deterministic forcing term and use the vorticity formulation of the equation
\begin{equation}\label{NS-vort}
{\rm d} \psi(t)+[-\nu \Delta \psi(t) +u(t)\cdot \nabla \psi(t)]\ {\rm d}t
= \tilde G{\rm d}W(t)
\end{equation}
where the vorticity is $\psi=\partial_2 u_1-\partial_1 u_2$
The bilinear term can be expressed as $\tilde B(\psi,\psi)$ by means of the Biot-Savart formula expressing the velocity $u$ in terms of the vorticity $\psi$. There is no deterministic forcing term.

The  noise is additive and described as follows.  Set
$\mathbb{Z}^2\setminus \{(0,0)\}= \mathbb{Z}_+^2 \cup \mathbb{Z}^2_-$, 
$\mathbb{Z}^2_+=\{\textbf{k}=(k_1,k_2) \in \mathbb{Z}^2 \ : \ k_2>0 \} \cup \{(k_1,0) \in \mathbb{Z}^2 \ : \ k_1>0\}$  and $\mathbb{Z}^2_-=\{(k_1,k_2) \in \mathbb{Z}^2 \ : \ -\textbf{k} \in \mathbb{Z}^2_+ \}$.
We  denote by 
\[
h_{\textbf{k}}(\xi):= 
\begin{cases}
\sin(\textbf{k} \cdot \xi) & \textbf{k} \in \mathbb{Z}^2_+,
\\
\cos(\textbf{k} \cdot \xi) & \textbf{k} \in \mathbb{Z}^2_-,
\end{cases}
\]
the basis for the space $\tilde H\subset L^2(\mathbb{T}^2)$  of periodic  real-valued square-integrable functions with vanishing mean. There are the eigenfunctions of the Laplace operator, with
 eigenvalues $\lambda_{\textbf{k}}=|\textbf{k}|^2$. The noise can be represented as
 $\tilde G{\rm d}W(t)=\sum_{\textbf{k}} \tilde G h_{\textbf{k}} {\rm d}\beta_{\textbf{k}} (t)$.

The main result in \cite{HaiMat} consists in proving the uniqueness of the invariant measure when only few modes are activated. More precisely, let $\mathcal{Z}_0$ be a  finite dimensional subset of the lattice 
$\mathbb{Z}^2\setminus \{(0,0)\}$ and assume 
\begin{equation}\label{G-hairer}
\tilde G h_{\textbf{k}}= 0  \text{ when } \textbf{k}\notin \mathcal{Z}_0,
\qquad 
\tilde G h_{\textbf{k}} \neq 0 \text{ when } \textbf{k} \in \mathcal{Z}_0 .
\end{equation}
This is the main result in \cite{HaiMat}.
\begin{theorem}
\label{HMthm}
Assume the noise term fulfills \eqref{G-hairer}, where 
$\mathcal{Z}_0$ is a symmetric finite dimensional set such that
\begin{itemize}
    \item [i)] there exist at least two elements in $\mathcal{Z}_0$ with different Euclidean norms,
    \item [ii)] integer linear combinations of elements of $\mathcal{Z}_0$ generate $\mathbb{Z}^2$.
\end{itemize}
Then, the Navier-Stokes equation in the vorticity formulation \eqref{NS-vort} has a unique invariant measure.
\end{theorem}

For instance, Theorem \ref{HMthm} holds for 
 $\mathcal{Z}_0=\{(1,0), (-1,0), (1,1), (-1,-1)\}$. 
 Notice that Theorem \ref{HMthm} gives a minimal nondegeneracy condition, independent of the viscosity $\nu$ and the intensity of the noise.

\begin{remark}
    We are not aware of results considering a multiplicative noise satisfying Assumptions \ref{noise1}-\ref{noise2} in the hypoelliptic setting.
\end{remark}

\section{Asymptotic stability of the invariant measure}
\label{sec_asy_sta}

Once the uniqueness of the invariant measure is proved, the natural question is whether this measure represents the statistical equilibrium of the system, that is if the law of the solution process asymptotically converges to it whatever is the initial velocity $u_0$.  In the deterministic setting, i.e. when $G=0$, for large enough viscosity  there is convergence with exponential decay (see, e.g., \cite[Theorem 10.2]{Temam1983}). In the stochastic setting,  there is a vast  literature concerning results of asymptotic stability of the invariant measure  in the case of an additive noise; however,  in the case of a multiplicative noise there are very few results. The only one dealing with the case of a not bounded multiplicative noise satisfying Assumptions \ref{noise1}-\ref{noise2} is our previous paper \cite{FerZanSNS}. 
There, we appeal to the generalized couplings techniques developed in \cite{KS} to infer the asymptotic stability of the unique invariant measure. 

The results obtained in \cite{KS} are a refinement of the results in \cite{GHMR17}; thus the framework is the one considered in Section \ref{effect} (effectively elliptic setting). 

\begin{theorem}
\label{KS}
Suppose that  the transition semigroup  associated to \eqref{sns} is a Feller semigroup on $H$ and for any $x, \tilde x \in H$ there exists some $\xi:=\xi_{x, \tilde x} \in \widehat{\mathcal{C}}(\mathbb{P}_{x}, \mathbb{P}_{\tilde x})$ such that $\pi_1(\xi)\sim \mathbb{P}_{x}$ and, for any $\varepsilon>0$,
\begin{equation}
\label{hyp_KS}
\lim_{n \rightarrow \infty} \xi \left((u, \tilde u) \in H^\mathbb{N} \times H^{\mathbb{N}} \ : \ 
\|u(n) - \tilde u(n)\|_H\le \varepsilon\right)=1.
\end{equation}
Then, there exists at most one invariant probability measure $\mu$ and
\begin{equation*}
{\cal L}(u(n;u_0)) 
\rightharpoonup \mu \quad \text{as} \ n \rightarrow \infty, \quad \forall \ x \in H.
\end{equation*}

%
\end{theorem}

Theorem \ref{KS} have been used in \cite{KS} to infer the (uniqueness  and) asymptotic stability of the invariant measure for system \eqref{sns} driven by an additive noise, working under the same non-degeneracy assumptions of Theorem \ref{GHMinv}. In \cite{FerZanSNS} we extend the result to the case of a multiplicative noise.
\begin{theorem}
\label{asy_thm}
Under the same Assumptions of Theorem \ref{unique_thm}, \eqref{sns} possesses at most one ergodic invariant measure $\mu\in \mathcal{P}(H)$ and
\begin{equation*}
{\cal L}(u(n;u_0))  \rightharpoonup \mu \quad \forall \ u_0 \in H.
\end{equation*}
\end{theorem}
The proof of the above result is similar to the proof of Theorem \ref{GHMinv} and relies on the Foias-Prodi estimates of Theorem \ref{FPlem}. For any $u_0, v_0 \in H$ we construct the measure $\xi_{u_0,v_0}$ on $H^{\mathbb{N}} \times H^{\mathbb{N}}$ given by the law of $(u(n), v(n))_{n \in \mathbb{N}}$, where $u$ solves \eqref{sns} with corresponding initial datum $u_0$ and $v$ solves \eqref{snsnud} with corresponding initial datum $v_0$, where we consider $N \ge \bar N$ with $\bar N$ as in Lemma \ref{FPlem}. As showed in Subsection \ref{effect}, provided Assumption \ref{noise3} holds with $M \ge N$, $\pi_2(\xi_{u_0,v_0}) \ll \mathbb{P}_{v_0}$. Moreover, $\pi_1(\xi_{u_0,v_0})= \mathbb{P}_{u_0}$. Exploiting the Foias-Prodi estimates, for any $\varepsilon>0$, we infer 
\begin{equation*}
\mathbb{P}\left(\|u(n)-v(n)\|^2_H >\varepsilon \right) \le \frac{1}{\varepsilon}\mathbb{E}\left[ \|u(t)-v(t)\|^2_H \right] \le \frac{1}{\varepsilon}\dfrac{C}{n^p},  \quad \forall  \ p \in  \left(0,\frac{\nu \lambda_1}{ 2 C_1}-\frac 38\right),
\end{equation*}
with $C$ a positive constants depending on the parameters of the equations. 
Therefore, 
\begin{equation}
\lim_{n \rightarrow \infty} \xi \left((u, v) \in H^\mathbb{N} \times H^{\mathbb{N}} \ : \ 
\|u(n) - v(n)\|_H\le \varepsilon\right) = \lim_{n \rightarrow \infty}\mathbb{P}\left(\|u(n)-v(n)\|^2_H \le\varepsilon \right) =1
\end{equation}
and this concludes the proof of the asymptotic stability of the unique invariant measure by appealing to Theorem \ref{KS}.

 \begin{remark}
In \cite{FerZanSNS} we prove an analogous version of Theorem \ref{asy_thm} in the case of a noise which is bounded or satisfies a sublinear growth condition. 
In these cases the result of Theorem \ref{asy_thm} holds for any $\nu>0$. The case of a bounded multiplicative noise is treated also in \cite{Oda2008}, always in an effectively elliptic setting. There, by means of coupling techniques, is also proved an exponential mixing result; see also \cite{Mat02b}.
\end{remark}

\begin{remark}
For the case of an additive noise there are plenty of works concerning the asymptotic stability of the unique invariant measure. The majority of the works also study the speed of convergence. For instance, we mention \cite{GolMas05} for the elliptic setting,
\cite{KS02}, \cite{KS}, \cite{BKS} for the effectively elliptic setting,  \cite{HaiMat2} and \cite{HaiMat3} for the hypoelliptic setting. 
\end{remark}


\section*{Acknowledgements}

The authors are members of Gruppo Nazionale per l’Analisi Matematica, la Probabilità e le loro Applicazioni (GNAMPA) of the Istituto Nazionale di Alta Matematica (INdAM). M.Z. gratefully acknowledges financial support through the project CUP-E53C23001670001.

M.Z. gratefully acknowledges the financial support of the project  “Prot. P2022TX4FE\_02 -  “Stochastic particle-based anomalous reaction-diffusion models with heterogeneous interaction for radiation therapy"  financed by the European Union - Next Generation EU, Missione 4-Componente 1-CUP: D53D23018970001.

B.F. gratefully acknowledges the financial support of the 
CAMRisk (Centre for the Analysis and Measurement of Global Risks)
at the Department of Economics and Management, University of Pavia.


\end{document}